\documentclass[12pt,a4paper]{article}
\usepackage{amssymb,amsmath}
\usepackage{mathrsfs}
\usepackage{tikz}
\usepackage{graphicx}
\usepackage{amsthm}

\theoremstyle{definition} 
\newtheorem{df}{Definition}[section] 
\newtheorem{exm}[df]{Example}   
\theoremstyle{plain}            
\newtheorem{pro}[df]{Proposition}
\newtheorem{lem}[df]{Lemma}
\newtheorem{theo}[df]{Theorem}
\newtheorem{cor}[df]{Corollary}

\newcommand{\f}{\ensuremath{\varphi}}

\newcommand{\al}{\ensuremath{\alpha}}

\newcommand{\la}{\ensuremath{\lambda}}

\newcommand{\Dcal}{\ensuremath{\mathcal{D}}}
\newcommand{\Ecal}{\ensuremath{\mathcal{E}}}

\newcommand{\Hcal}{\ensuremath{\mathcal{H}}}
\newcommand{\Kcal}{\ensuremath{\mathcal{K}}}
\newcommand{\Mcal}{\ensuremath{\mathcal{M}}}
\newcommand{\Ncal}{\ensuremath{\mathcal{N}}}

\newcommand{\Zcal}{\ensuremath{\mathcal{Z}}}

\newcommand{\rr}{\ensuremath{\mathbb{R}}}

\newcommand{\unit}{\ensuremath{\mathbf{1}}}
\newcommand{\norm}[1]{\ensuremath{\left\|#1\right\|}}

\newcommand{\set}[2]{\left\{#1\,\middle|\,#2 \right\}}

\newcommand{\Ca}{$C${\rm*}-algebra}

\newcommand{\vNa}{von Neumann algebra}

\newcommand{\AW}{$AW${\rm*}-algebra}
\newcommand{\AWf}{$AW${\rm*}-factor}
\newcommand{\AWsa}{$AW${\rm*}-subalgebra}

\newcommand{\Sic}{{\rm*}-isomorphic}
\newcommand{\Jsi}{Jordan {\rm*}-isomor\-phism}

\newcommand{\ifff}{if and only if}

\begin{document}

\begin{center}
{\Large \bf Spectral order isomorphisms and \AWf{}s}\\

{\large Martin Bohata\footnote{bohata@math.feld.cvut.cz}\\}
         \it Department of Mathematics, Faculty of Electrical Engineering,\\
        Czech Technical University in Prague, Technick\'a 2,\\ 
        166 27 Prague 6, Czech Republic        
\end{center}

{\small \textbf{Abstract:}
The paper deals with spectral order isomorphisms in the framework of \AW{}s. We establish that every spectral order isomorphism between sets of all self-adjoint operators (or between sets of all effects, or between sets of all positive operators) in \AWf{}s of Type I has a canonical form induced by a continuous function calculus and an isomorphism between projection lattices. In particular, this solves an open question about spectral order automorphisms of the set of all (bounded) self-adjoint operators on an infinite-dimensional Hilbert space. We also discuss spectral order isomorphisms preserving, in addition, orthogonality in both directions.
} 


\section{Introduction}

Self-adjoint operators in a \vNa{} can be equipped with a number of partial orders. The most common is L\"{o}wner order induced by positive operators. If a \vNa{} under consideration is abelian, then its self-adjoint part endowed with L\"{o}wner order forms a conditionally complete lattice. However, this is no longer true when we consider noncommutative \vNa{}s \cite{Ka51,Sh51}.

There is a natural question whether one can find a partial order on the self-adjoint part of a \vNa{} which has lattice properties and coincides with L\"{o}wner order in the abelian case. An affirmative answer was given by Olson \cite{Ol71} who introduced so called spectral order by means of spectral families in \vNa{}s. The definition of the spectral order can be directly extended to the more general framework of \AW{}s as follows. The {\it spectral order} is a binary relation $\preceq$ on the self-adjoint part of an \AW{} defined by setting $x\preceq y$ if $E_\la^y\leq E_\la^x$ for each $\la\in\rr$, where $(E^x_\la)_{\la\in\rr}$ and $(E^y_\la)_{\la\in\rr}$ are spectral families of self-adjoint operators $x$ and $y$, respectively, and $\leq$ is the L\"{o}wner order.

The spectral order has been especially studied on \vNa{}s (see, for example, \cite{Bo19,dG05} and references therein). Nevertheless, it has been considered also in contexts of \AW{}, unbounded operators, and further structures \cite{FP19,Ha07,HT16,PS12,Tu14}. From the physical point of view, the spectral order is significant in a categorical formulation of quantum theory \cite{DD14,Ha11,Wo14}.

The main aim of this paper is the study of isomorphisms between posets of self-adjoint operators equipped with the spectral order. Let $M$ and $N$ be sets of self-adjoint elements in \AW{}s $\Mcal$ and $\Ncal$, respectively. We say that a bijection $\f:M\to N$ is a {\it spectral order isomorphism} if, for all $x,y\in M$, $x\preceq y$ \ifff{} $\f(x)\preceq \f(y)$. In the special case $M=N$, a spectral order isomorphism $\f$ is called {\it spectral order automorphism}. In this paper, we show that every spectral order isomorphism between sets of all effects (i.e. positive operators in the unit ball) in \AWf{}s of Type I has the form $\f(x)=\Theta_\tau(f(x))$, where $f:[0,1]\to[0,1]$ is a strictly increasing bijection, $\tau$ is an isomorphism between projection lattices, and $\Theta_\tau$ is defined by $E^{\Theta_\tau(x)}_\la=\tau(E^{x}_\la)$ for all $\la\in\rr$. Let us remark that isomorphisms of the above form were studied for the first time by Turilova in \cite{Tu14}. As each \AWf{} of Type I is precisely a \vNa{} of all bounded operators on a Hilbert space, our result is a direct extension of a complete description of spectral order automorphisms of the set of all effects on a Hilbert space found by Moln\'{a}r, Nagy, and \v{S}emrl \cite{MN16,MS07}. We also obtain a similar form for all spectral order isomorphisms between self-adjoint parts (as well as between positive parts) of \AWf{}s of Type I. This proves the conjecture of Moln\'{a}r and \v{S}emrl \cite{MS07} about spectral order automorphisms on the set of all (bounded) self-adjoint operators on an infinite-dimensional Hilbert space. Moreover, our approach also covers finite-dimensional case with much simpler proof than that of \cite{MN16} and \cite{MS07}.

The last part of this paper is devoted to spectral order orthoisomorphisms. By a {\it spectral order orthoisomorphism}, we mean a spectral order isomorphism $\f:M\to N$ preserving orthogonality in both directions (i.e., for all $x,y\in M$, $xy=0$ \ifff{} $\f(x)\f(y)=0$). A spectral order orthoisomorphism is called a {\it spectral order orthoautomorphism}, if $M=N$. It was proved by Hamhalter and Turilova in \cite{HT17} that if $\f$ is a spectral order orthoautomorphism of the set of all effects in an \AWf{} that is not of Type III and of Type I$_2$, then $\f(x)=\psi(f(x))$, where $\psi$ is a Jordan {\rm*}-automorphism and $f:[0,1]\to[0,1]$ is a strictly increasing bijection. We generalize this theorem to spectral order orthoisomorphisms between sets of all effects in arbitrary \AWf{}s that are not of Type I$_2$. Moreover, we show that a similar statement is also true when we consider spectral order orthoisomorphisms between self-adjoin parts (or between positive parts) of \AWf{}s.
\section{Preliminaries}

Let us start this section with summarizing some basic facts about \AW{}s. For a detailed treatment of \AW{}s, we refer the reader to \cite{Be11}. An {\it \AW{}} is a \Ca{} \Mcal{} such that for each nonempty set $S\subseteq\Mcal$ there is a projection $p\in\Mcal$ satisfying 
$$
\set{x\in\Mcal}{sx=0 \mbox{ for all } s\in S}=p\Mcal.
$$
It is well known that every von Neumann algebra is \AW{} but the converse is not true. More generally, every monotone complete \Ca{} is an \AW{}. However, it is an open question whether class of \AW{}s coincides with a class of monotone complete \Ca{}s \cite{SW15}. Each \AW{} is unital and is closed linear span of all its projections. The center
$$
\Zcal(\Mcal)=\set{x\in\Mcal}{xy=yx \mbox{ for all }y\in\Mcal}
$$
of an \AW{} \Mcal{} forms an \AWsa{} of \Mcal{}. An {\it \AWf{}} is an \AW{} with one-dimensional center. A classification of \AW{}s is analogous to that of \vNa{}s. Kaplansky \cite{Ka52} proved that an \AW{} of Type I is a \vNa{} \ifff{} its center is a \vNa{}. In particular, every \AWf{} of Type I is a von Neumann factor of Type I. In the sequel, we shall denote by $P(\Mcal)$ the set of all projections in an \AW{} $\Mcal$. The set $P(\Mcal)$ (endowed with the L\"{o}wner order) forms a complete lattice called the {\it projection lattice}. By $\Mcal_{sa}$ and $\Ecal(\Mcal)$, we shall denote the self-adjoint part of \Mcal{} and the set of all effects (i.e. positive operators in the unit ball of \Mcal{}), respectively.

From the point of view of abelian algebras, \AW{}s are more natural than \vNa{}s. The reason is that $C(X)$ is an \AW{} \ifff{} $X$ is a Stonean space (i.e. an extremally disconnected compact Hausdorff topological space). Thus every complete Boolean algebra is isomorphic to the projection lattice of an abelian \AW{}. On the other hand, the projection lattice of each abelian \AW{} is a complete Boolean algebra. Consequently, the projection lattice $(P(\Mcal),\leq)$ of an abelian \AW{} \Mcal{} is meet-infinitely distributive (see, for example, \cite[Theorem 5.13]{Ro08}). More specifically,
$$
\bigwedge_{\la\in\Lambda}(p\vee q_\la)=p\vee \bigwedge_{\la\in\Lambda}q_\la
$$
for each $e\in P(\Mcal)$ and each family $(q_\la)_{\la\in\Lambda}$ in $P(\Mcal)$. Note also that a nonempty set $S$ of mutually commuting self-adjoint operators in an \AW{} \Mcal{} generates an abelian \AWsa{} of \Mcal{}.

Let us look at an auxiliary result concerning projections.

\begin{lem}\label{supremum of infima}
	Let $I\subseteq\rr$ be an interval. If $(p_\la)_{\la\in I},(q_\la)_{\la\in I}$ are increasing nets of projections in an abelian \AW{}, then
	$$
	\bigwedge_{\la\in I}p_\la\vee q_\la =\left(\bigwedge_{\la\in I}p_\la\right)\vee \left(\bigwedge_{\mu\in I} q_\mu\right).
	$$
\end{lem}
\begin{proof}
	It is easy to see that $\bigwedge_{\la\in I}p_\la\vee q_\la\geq\bigwedge_{(\la,\mu)\in I^2}p_\la\vee q_\mu$. On the other hand, $\bigwedge_{\la\in I}p_\la\vee q_\la\leq\bigwedge_{(\la,\mu)\in I^2}p_\la\vee q_\mu$ because $p_\la\vee q_\mu\geq p_\nu\vee q_\nu\geq \bigwedge_{\la\in I}p_\la\vee q_\la$ whenever $\la,\mu\in I$ and $\nu=\min\{\la,\mu\}$. This establishes that 
	$$
	\bigwedge_{\la\in I}p_\la\vee q_\la=\bigwedge_{(\la,\mu)\in I^2}p_\la\vee q_\mu.
	$$
	
	As the projection lattice of each abelian AW*-algebra is meet-infinitely distributive, we have
	$$
	\bigwedge_{\la\in I}p_\la\vee q_\la=\bigwedge_{(\la,\mu)\in I^2}p_\la\vee q_\mu=\bigwedge_{\la\in I}\bigwedge_{\mu\in I} p_\la\vee q_\mu =\left(\bigwedge_{\la\in I} p_\la\right)\vee\left(\bigwedge_{\mu\in I} q_\mu\right).
	$$
\end{proof}

Let $\Mcal,\Ncal$ be \AW{}s. We say that $\f:P(\Mcal)\to P(\Ncal)$ is {\it orthoisomorphism} if it is a bijection preserving orthogonality in both directions. Dye's theorem \cite{Dy55} is an important extension theorem for orthoisomorphisms between projection lattices of \vNa{}s. The following generalization of this famous theorem to the context of \AW{}s was proved by Hamhalter in \cite{Ha15}.

\begin{theo}[{\cite[Theorem~4.3]{Ha15}}]\label{Dye theorem}
	Let $\Mcal$ and $\Ncal$ be \AW{}s and let \Mcal{} have no Type I$_2$ direct summand. If $\f:P(\Mcal)\to P(\Ncal)$ is an orthoisomorphism, then there is a (unique) \Jsi{} $\Phi:\Mcal\to\Ncal$ extending \f.
\end{theo}

The next proposition characterizes central projections in terms of lattice distributivity.

\begin{pro}[{\cite[Proposition~3.1]{HT17}}]\label{characterizations of central projections}
	If \Mcal{} is an \AW{} and $z\in P(\Mcal)$, then following conditions are equivalent:
		\begin{enumerate}
		\item $z$ is central.
		\item $z=(z\wedge p)\vee(z\wedge (\unit-p))$ for each $p\in P(\Mcal)$.
		\item $z\wedge (p\vee q)=(z\wedge p)\vee(z\wedge q)$ for all $p,q\in P(\Mcal)$.
	\end{enumerate}
\end{pro}

Now we turn our attention to spectral families. Recall that a family $(E_\la)_{\la\in\rr}$ of projections in an \AW{} $\Mcal$ is called a {\it spectral family} if the following conditions hold:
\begin{enumerate}
	\item $E_\la\leq E_\mu$ whenever $\la\leq \mu$.
	\item $E_\la=\bigwedge_{\mu>\la}E_\mu$ for every $\la\in\rr$.
	\item There is a positive real number $\al$ such that $E_\la=0$ when $\la<-\al$ and $E_\la=\unit$ when $\la>\al$, where \unit{} is the unit of \Mcal{}.
\end{enumerate}
It is well known that there exists a bijection between set of all spactral families in \Mcal{} and the self-adjoint part of \Mcal{}. The spectral family $(E_\la)_{\la\in\rr}$ corresponds to $x\in\Mcal_{sa}$ \ifff{} $xE_\la\leq \la E_\la$ and $\la(\unit-E_\la)\leq x(\unit-E_\la)$ for each $\la\in\rr$. In the sequel, we shall denote by $(E^x_\la)_{\al\in\rr}$ the spectral family corresponding to $x\in\Mcal_{sa}$.  It is worthwhile to point out that all elements of $(E^x_\la)_{\al\in\rr}$ belong to the abelian \AWsa{} of \Mcal{} generated by $\{\unit,x\}$.

\begin{lem}\label{rescaling and translation}
	Let $x$ be a self-adjoint element in an \AW{}.
	\begin{enumerate}
		\item If $\al$ is a positive real number, then $E^{\al x}_\la=E^x_{\frac{\la}{\al}}$ for all $\la\in\rr$.
		\item If $\al\in\rr$, then $E^{x+\al\unit}_\la= E^{x}_{\la-\al}$ for all $\la\in\rr$
	\end{enumerate}
\end{lem}
\begin{proof}
The statements follow immediately from elementary computations.
\end{proof}

\begin{lem}\label{spectral families in different algebras}
	Let \Ncal{} be an \AWsa{} of an \AW{} \Mcal{}. Suppose that $\unit_\Mcal$ and $\unit_\Ncal$ are units of \Mcal{} and \Ncal{}, respectively. If $x\in\Ncal_{sa}$ and $(E^x_\la)_{\la\in\rr}$ is  the spectral family of $x$ in \Mcal{}, then $(\unit_{\Ncal} E^x_\la)_{\la\in\rr}$ is the spectral family of $x$ in \Ncal{} and 
	$$
	(\unit_\Mcal-\unit_\Ncal)E^x_\la=\begin{cases}0&\quad\mbox{ if }\la<0;\\ \unit_\Mcal-\unit_\Ncal&\quad\mbox{ if }\la\geq 0.\end{cases}
	$$
\end{lem}
\begin{proof}
	As all elements of $(E^x_\la)_{\la\in\rr}$ are in the abelian \AW{} generated by $\{x,\unit_\Mcal,\unit_\Ncal\}$, $\unit_\Ncal$ commutes with $E^x_\la$ for every $\la\in\rr$. Thus $\unit_{\Ncal} E^x_\la\in P(\unit_\Ncal\Mcal\unit_\Ncal)$ and 
	$$
	(\unit_\Mcal-\unit_\Ncal)E^x_\la\in P((\unit_\Mcal-\unit_\Ncal)\Mcal(\unit_\Mcal-\unit_\Ncal)).
	$$
	If $\la\in\rr$, then
	\begin{eqnarray*}
		x\unit_{\Ncal} E^x_\la 
			&\leq & \la \unit_{\Ncal} E^x_\la,\\
		x(\unit_{\Ncal} -\unit_{\Ncal} E^x_\la)
			&\geq& \la(\unit_{\Ncal} -\unit_{\Ncal} E^x_\la),\\
		(\unit_\Mcal-\unit_\Ncal)x[(\unit_\Mcal-\unit_\Ncal)E^x_\la]
			&\leq&\la[(\unit_\Mcal-\unit_\Ncal)E^x_\la],\\
		(\unit_\Mcal-\unit_\Ncal)x[(\unit_\Mcal-\unit_\Ncal)-(\unit_\Mcal-\unit_\Ncal)E^x_\la]
			&\geq& \la [(\unit_\Mcal-\unit_\Ncal)-(\unit_\Mcal-\unit_\Ncal)E^x_\la].
	\end{eqnarray*}
	Thus $(\unit_{\Ncal} E^x_\la)_{\la\in\rr}$ is the spectral family of $x$ in \Ncal{} (because $\Ncal\subseteq \unit_\Ncal\Mcal\unit_\Ncal$ and the units of $\unit_\Ncal\Mcal \unit_\Ncal$ and \Ncal are same) and $((\unit_\Mcal-\unit_\Ncal) E^x_\la)_{\la\in\rr}$ is the spectral family of $(\unit_\Mcal-\unit_\Ncal)x=0$ in $(\unit_\Mcal-\unit_\Ncal)\Mcal{}(\unit_\Mcal-\unit_\Ncal)$. Hence
	$$
	(\unit_\Mcal-\unit_\Ncal)E^x_\la=\begin{cases}0&\quad\mbox{ if }\la<0;\\ \unit_\Mcal-\unit_\Ncal&\quad\mbox{ if }\la\geq 0.\end{cases}
	$$
\end{proof}

At the end of this section, we consider some examples of spectral order isomorphisms. We shall see later that they are building blocks of a number of spectral order isomorphisms.  

\begin{exm}\label{spectral order isomorhisms and function calculus}
	Consider strictly increasing bijection $f:\rr\to\rr$. Suppose that $x$ is a self-adjoint operator in an \AW{} \Mcal{} and $C(X)$ is an abelian \AW{} \Sic{} to the \AW{} generated by $\{x,\unit\}$. Let $g\in C(X)$ correspond to $x$. The functions $E^g_\la$ and $E^{f\circ g}_\la$ are characteristic functions of clopen sets $X_\la=X\setminus\overline{g^{-1}((\la,\infty))}$ and 
	$$
	Y_\la=X\setminus\overline{(f\circ g)^{-1}((\la,\infty))}=X\setminus\overline{g^{-1}((f^{-1}(\la),\infty))}=X_{f^{-1}(\la)}, 
	$$
	respectively. This shows that $E^{f(x)}_\la=E^x_{f^{-1}(\la)}$ for every $\la\in\rr$. Using the definition of spectral order we conclude that $\f:\Mcal_{sa}\to\Mcal_{sa}$ given by $\f(x)=f(x)$ is a spectral order automorphism. 
	
	Note that $f(0)=0$ \ifff{} $f$ preserves orthogonality of self-adjoint elements in both directions. This can be shown by passing to the abelian \AW{} $C(X)$ \Sic{} to the \AW{} generated by $\{\unit,x,y\}$. Let $g_1$ and $g_2$ be functions in $C(X)$ corresponding to $x$ and $y$, respectively. Suppose that $f(0)=0$. Then $f^{-1}(0)=0$ and so $g_1g_2=0$ \ifff{} $(f\circ g_1)(f\circ g_2)=0$. It remains to prove that $f(0)=0$ whenever $f$ preserves orthogonality. To see this let $f(0)\neq 0$. As $0$ is orthogonal to itself, we have $f(0)f(0)=0$ which is a contradiction. A simple consequence of just proved equivalence says that the map \f{} defined above is a spectral order orthoautomorphism \ifff{} $f(0)=0$.
	
	In a similar way, we can obtain spectral order automorphisms of $\Ecal(\Mcal)$ and $\Mcal_+$ by considering strictly increasing bijections on $[0,1]$ and $[0,\infty)$, respectively. In these cases, the strictly increasing bijection $f$ automatically satisfies $f(0)=0$ and so every spectral order automorphism of given form is a spectral order orthoautomorphism. 
\end{exm}

Let $\Mcal$ and $\Ncal$ be \AW{}s. A bijection $\tau:P(\Mcal)\to P(\Ncal)$ is called a {\it projection isomorphism} if it preserves the L\"{o}wner order in both directions (i.e., for all $p,q\in P(\Mcal)$, $p\leq q$ \ifff{} $\tau(p)\leq \tau(q)$). The structure of all projection isomorphisms is well known in the case of $\Mcal=B(\Hcal)$ and $\Ncal=B(\Kcal)$, where \Hcal{} and \Kcal{} are Hilbert spaces of dimension at least 3. (As usual, $B(\Hcal)$ denotes the \vNa{} of all bounded operators on a Hilbert space \Hcal{}.) Indeed, let $L(\Hcal)$ and $L(\Kcal)$ be lattices of all closed subspaces of \Hcal{} and \Kcal{}, respectively. There is a natural correspondence between $\tau$ and an lattice isomorphism $\sigma: L(\Hcal)\to L(\Kcal)$. If \Hcal{} and \Kcal{} are finite-dimensional spaces (of dimension at least 3), then the fundamental theorem of projective geometry (see, for example, \cite[p. 203]{Be95}) says that $\sigma(X)=sX$ for some semi-linear bijection $s$ from \Hcal{} onto \Kcal{}. In the infinite-dimensional case, $\sigma(X)=sX$, where $s$ is a linear or conjugate-linear bijection from \Hcal{} onto \Kcal{} \cite{FL84}. 

\begin{exm}
	Given a projection isomorphism $\tau:P(\Mcal)\to P(\Ncal)$ we define a map $\Theta_\tau: \Mcal_{sa}\to \Ncal_{sa}$ by 
	$$
	E^{\Theta_\tau(x)}_\la=\tau(E^x_\la)
	$$
	for all $x\in \Mcal_{sa}$ and $\la\in\rr$. It is easy to verify that $\Theta_\tau$ is a spectral order isomorphism.
	
	By restrictions, we get spectral order isomorphisms between positive parts and between sets of all effects. 
\end{exm}

In the sequel, $\Theta_\tau$ will always denote just described spectral order isomorphism determined by a projection isomorphism $\tau$. 

Motivated by the previous examples, we say that a spectral order isomorphism $\f:\Ecal(\Mcal)\to\Ecal(\Ncal)$ (respectively $\f:\Mcal_+\to\Ncal_+$, $\f:\Mcal_{sa}\to\Ncal_{sa}$) is {\it canonical} if there are a strictly increasing bijection $f:[0,1]\to[0,1]$ (respectively $f:[0,\infty)\to[0,\infty)$, $f:\rr\to\rr$) and a projection isomorphism $\tau: P(\Mcal)\to P(\Ncal)$ such that 
$$
\f(x)=\Theta_\tau(f(x))
$$ 
for all $x\in\Ecal(\Mcal)$ (respectively $x\in\Mcal_+$, $x\in\Mcal_{sa}$).

Now we give an example of a spectral order orthoisomorphism of the form $\Theta_\tau$. Before doing this, we recall that a linear bijection $\psi:\Mcal\to\Ncal$ between \AW{}s \Mcal{} and \Ncal{} is called a {\it \Jsi{}} if $\psi(x^2)=\psi(x)^2$ and $\psi(x^*)=\psi(x)^*$ for all $x\in\Mcal$. 

\begin{exm}
	Let $\psi:\Mcal\to \Ncal$ be a \Jsi{} between \AW{}s of \Mcal{} and \Ncal{}. We observe that 
	$$
	E^{\psi(x)}_\la=\psi(E^{x}_\la)
	$$ 
	for all $\la\in\rr$ whenever $x\in\Mcal_{sa}$. If a projection isomorphism $\tau$ is a restriction of $\psi$, then $\Theta_\tau(x)=\psi(x)$ for all $x\in\Mcal_{sa}$. Moreover, $\psi$ preserves orthogonality relation in both directions and so a restriction of $\psi$ is a spectral order orthoisomorphism.
\end{exm}
\section{Spectral order on AW*-algebras}

We begin this section with simple observations about spectral order. Statements in the following lemma were proved for \vNa{}s in papers \cite{dG05,Ol71}.

\begin{lem}\label{basic properties of specral order}
	Let \Mcal{} be an \AW{}. Assume that $x,y\in\Mcal_{sa}$, $\al,\beta\in\rr$, and $\al>0$.
	\begin{enumerate}
		\item $x\preceq y$ \ifff{} $\al x+\beta\unit\preceq \al y+\beta\unit$.
		\item The spectral order $\preceq$ coincides with the L\"{o}wner order $\leq$ on $P(\Mcal)$.
		\item If $x\preceq y$, then $x\leq y$.
		\item Suppose that $x$ commutes with $y$. Then $x\preceq y$ \ifff{} $x\leq y$.
	\end{enumerate}
\end{lem}
\begin{proof}
Using Lemma~\ref{rescaling and translation}, we immediately obtain the desired equivalence in (i).

By considering spectral families of $p,q\in P(\Mcal)$, $p\preceq q$ \ifff{} $p\leq q$. This shows (ii).

It remains to prove (iii) and (iv). Set $\al=\max\{\norm{x},\norm{y}\}$. If $\al=0$, then both statements are clear. Assume that $\al>0$. Then 
$$
\frac{1}{2\al}x+\frac{1}{2}\unit,\frac{1}{2\al}y+\frac{1}{2}\unit\in\Ecal(\Mcal).
$$
Furthermore, we see from (i) that 
$$
x\preceq y\Leftrightarrow \frac{1}{2\al}x+\frac{1}{2}\unit\preceq\frac{1}{2\al}y+\frac{1}{2}\unit.
$$
Now it follows from \cite[Corollary~1]{HT16} and \cite[Corollary~2]{HT16} that (iii) and (iv) are true.
\end{proof}

\begin{pro}
	Let \Ncal{} be an \AWsa{} of an \AW{} \Mcal{}. Suppose that $x,y\in \Ncal_{sa}$. Then $x\preceq y$ in $\Mcal_{sa}$ \ifff{} $x\preceq y$ in $\Ncal_{sa}$.
\end{pro}
\begin{proof}
	Let $(E^x_\la)_{\la\in\rr}$ and $(E^y_\la)_{\la\in\rr}$ be spectral families of $x$ and $y$, respectively, in \Mcal{}. If $x\preceq y$ in \Mcal{}, then $E^y_\la\leq E^x_\la$ for all $\la\in\rr$. Since the unit $\unit_\Ncal$ of \Ncal{} commutes with $E^x_\la$ and $E^y_\la$ for all $\la\in\rr$, $\unit_\Ncal{}E^y_\la\leq \unit_\Ncal{}E^x_\la$ for all $\la\in\rr$. It follows from Lemma~\ref{spectral families in different algebras} that $x\preceq y$ in \Ncal{}.
	
	Conversely, let $x\preceq y$ in \Ncal{}. Then $\unit_\Ncal{}E^y_\la\leq \unit_\Ncal{}E^x_\la$ for all $\la\in\rr$ by Lemma~\ref{spectral families in different algebras}. Moreover, Lemma~\ref{spectral families in different algebras} implies that 
	$$
	(\unit_\Mcal-\unit_\Ncal)E^x_\la=(\unit_\Mcal-\unit_\Ncal)E^y_\la
	$$
	for all $\la\in\rr$. Therefore,
	$$
	E^y_\la=\unit_\Ncal{}E^y_\la+(\unit_\Mcal-\unit_\Ncal)E^y_\la \leq \unit_\Ncal{}E^x_\la+(\unit_\Mcal-\unit_\Ncal)E^x_\la=E^x_\la
	$$
	for all $\la\in\rr$. Hence $x\preceq y$ in \Mcal{}.
\end{proof}

It can be proved that the self-adjoint part of an \AW{} equipped with the spectral order is a conditionally complete lattice (the proof is the same as for \vNa{}s \cite{Ol71}). Suprema and infima are determined by their spectral families as follows. Let \Mcal{} be an \AW{} and let $M\subseteq\Mcal_{sa}$ be nonempty. If $M$ is bounded above, then its supremum $\bigvee_{x\in M} x$ in $(\Mcal_{sa},\preceq)$ is a self-adjoint element with the spectral family
$$
\left(E^{\bigvee_{x\in M} x}_\la\right)_{\la\in\rr}
=\left(\bigwedge_{x\in M}E^{x}_\la\right)_{\la\in\rr}.
$$
If $M$ is bounded below, then its infimum $\bigwedge_{x\in M}x$ in $(\Mcal_{sa},\preceq)$ is a self-adjoint element with the spectral family
$$
\left(E^{\bigwedge_{x\in M} x}_\la\right)_{\la\in\rr}
=\left(\bigwedge_{\mu>\la}\bigvee_{x\in M}E^{x}_\mu\right)_{\la\in\rr}.
$$
The next two assertions state that suprema and infima with respect to the spectral order are stable under taking smaller posets. The latter is a generalization of \cite[Proposition~2.2]{Bo19} to the case of \AW{}s.

\begin{pro}\label{independence}
	Let \Ncal{} be an \AWsa{} of an \AW{} \Mcal{}. Suppose that $M\subseteq\Ncal_{sa}$ is nonempty. 
	\begin{enumerate}
		\item If $M$ has an upper bound in $(\Mcal_{sa},\preceq)$, then both the supremum of $M$ in $(\Mcal_{sa},\preceq)$ and the supremum of $M$ in $(\Ncal_{sa},\preceq)$ exist and coincide.
		\item If $M$ has a lower bound in $(\Mcal_{sa},\preceq)$, then both the infimum of $M$ in $(\Mcal_{sa},\preceq)$ and the infimum of $M$ in $(\Ncal_{sa},\preceq)$ exist and coincide.
	\end{enumerate}
\end{pro}
\begin{proof}
	\begin{enumerate}
		\item The existence of suprema is ensured by the fact that $(\Mcal_{sa},\preceq)$ and $(\Ncal_{sa},\preceq)$ are conditionally complete lattices. If $x\in\Mcal_{sa}$, then we denote by $(E^x_\la)_{\la\in\rr}$ the spectral family of $x$ in \Mcal{}. Let $a$ be the supremum of $M$ in \Mcal{} and let $b$ be the supremum of $M$ in \Ncal{}. By Lemma~\ref{spectral families in different algebras},
		\begin{eqnarray*}
			E^b_\la&=&\unit_\Ncal E^b_\la + (\unit_\Mcal{}-\unit_\Ncal )E^b_\la
			= \bigwedge_{x\in M} \unit_\Ncal E^x_\la + \bigwedge_{x\in M} (\unit_\Mcal{}-\unit_\Ncal )E^x_\la\\
			&=&\unit_\Ncal \bigwedge_{x\in M} E^x_\la + (\unit_\Mcal{}-\unit_\Ncal) \bigwedge_{x\in M} E^x_\la
			= \bigwedge_{x\in M} E^x_\la=E^a_\la
		\end{eqnarray*}			
		for all $\la\in\rr$. Thus $a=b$.	
		
		\item This follows from (i) by considering the order-reversing map $t\mapsto-t$.
	\end{enumerate}
\end{proof}

\begin{pro}\label{sublattices}
	Let \Mcal{} be an \AW{}. 
	\begin{enumerate}
		\item If $M\subseteq \Mcal_+$ is nonempty and bounded above in $(\Mcal_{sa},\preceq)$, then the supremum of $M$ in $(\Mcal_{sa},\preceq)$ is a positive element. 
		\item If $M\subseteq \Mcal_+$ is nonempty, then the infimum of $M$ in $(\Mcal_{sa},\preceq)$ is a positive element.
		\item If $L\in\left\{\Ecal(\Mcal),P(\Mcal)\right\}$, then the supremum and the infimum of every nonempty subset of $L$ in $(\Mcal_{sa},\preceq)$ belong to $L$.
	\end{enumerate}
\end{pro}
\begin{proof}
The proof is similar to that of \cite[Proposition~2.2]{Bo19} and so it is omitted.
\end{proof}

We shall denote suprema and infima of a set $M$ in posets $(\Mcal_{sa},\preceq)$, $(\Mcal_{+},\preceq)$, $(\Ecal(\Mcal),\preceq)$, and $(P(\Mcal),\leq)=(P(\Mcal),\preceq)$ by symbols $\bigvee_{x\in M} x$ and $\bigwedge_{x\in M} x$, respectively. We have seen in the previous result that this simple notation cannot lead to any possible misunderstanding.

\begin{lem}\label{supremum of orthogonal elements}
Let \Mcal{} be an \AW{}. If $x,y\in\Mcal_{sa}$ are mutually orthogonal, then 
	$$
	x\vee y=x^{+}+y^{+},
	$$
where $x^{+}$ and $y^{+}$ are the positive parts of $x$ and $y$, respectively.
\end{lem}
\begin{proof}
The statement clearly holds for $\Mcal=C(X)$, where $X$ is Stonean. Since $x$ commutes with $y$, we can reduce the general case to the above special case by Proposition~\ref{independence}.
\end{proof}

The content of the following lemma appeared in the paper \cite{MS07} without proof. Therefore, we present a proof for convenience of the reader.

\begin{lem}\label{supremum of multiples of projections}
	Let \Mcal{} be an \AW{} and let $p,q\in P(\Mcal)$. If real numbers $\al$ and $\beta$ satisfy $0\leq \al\leq \beta$, then 
	$\al p\vee \beta q=\al(p\vee q-q)+\beta q$.
\end{lem}
\begin{proof}
	The case $\al=0$ is clear because $0\vee \beta q=\beta q$ by Lemma~\ref{supremum of orthogonal elements}.
	
	If $0<\al=\beta$, then $\al p\vee \beta q=\al(p\vee q-q)+\beta q$ because multiplication by a positive scalar is a spectral order isomorphism (see Lemma~\ref{basic properties of specral order}).
	
	Suppose that $0<\al<\beta$. Using Lemma~\ref{rescaling and translation}, we can directly verify that
	$$
	E^{\al p\vee\beta q}_{\la}=\begin{cases}
	0, &\quad \la<0;\\
	(\unit-p)\wedge(\unit-q), &\quad \la\in[0,\al);\\
	(\unit-q), &\quad \la\in[\al,\beta);\\
	\unit, &\quad \la\geq \beta.
	\end{cases}
	$$
	As $q\leq p\vee q$, $p\vee q$ commutes with $\unit-q$ and so 
	$$
	p\vee q-q=(p\vee q)(\unit-q)=(p\vee q)\wedge(\unit-q).
	$$
	Moreover, $q$ is orthogonal to $p\vee q-q$. Consequently, 
	$$
	E^{\al(p\vee q-q)+\beta q}_\la=\begin{cases}
	0, &\quad \la<0;\\
	[((\unit-p)\wedge(\unit-q))\vee q]\wedge(\unit-q), &\quad \la\in[0,\al);\\
	(\unit-q), &\quad \la\in[\al,\beta);\\
	\unit, &\quad \la\geq \beta.
	\end{cases}
	$$
	Since $q$ is orthogonal to $(\unit-p)\wedge(\unit-q)$ and $\unit-q$ commutes with $[(\unit-p)\wedge(\unit-q)]+ q$, we have
	\begin{eqnarray*}
		[((\unit-p)\wedge(\unit-q))\vee q]\wedge(\unit-q)
		&=&[((\unit-p)\wedge(\unit-q))+ q]\wedge(\unit-q)\\
		&=&[((\unit-p)\wedge(\unit-q))+ q](\unit-q)\\
		&=&(\unit-p)\wedge(\unit-q).
	\end{eqnarray*}
	Hence 
	$$
	E^{\al p\vee\beta q}_{\la}=E^{\al(p\vee q-q)+\beta q}_\la
	$$
	for all $\la\in\rr$.
\end{proof}

A projection $p$ in an \AW{} \Mcal{} is said to be {\it atomic} if there is no nonzero projection $q$ in \Mcal{} such that $q\neq p$ and $q\leq p$. We shall denote by $P_{at}(\Mcal)$ the set of all atomic projections in an \AW{} \Mcal{}. Now we characterize effects corresponding nonzero scalar multiples of atomic projections in terms of the spectral order. 

\begin{pro}\label{characterization of atomic projections on effects}
	Let $\Mcal$ be an \AW{} and let $x\in \Ecal(\Mcal)$ be nonzero. Then the following statements are equivalent:
	\begin{enumerate}
		\item There are $\la\in (0,1]$ and an atomic projection $e\in \Mcal$ such that $x=\la e$.
		\item If $y,z\in\Ecal(\Mcal)$ satisfy $y,z\preceq x$, then $y\preceq z$ or $z\preceq y$.
	\end{enumerate}
\end{pro}
\begin{proof}
Assume that (i) holds and $y,z\in\Ecal(\Mcal)$ satisfy $y,z\preceq x$. Then $y,z\leq \la e$. Since $e\Mcal{}e$ is a hereditary \AWsa{} of $\Mcal$, we have $y,z\in e\Mcal{}e$. From the fact that $e\Mcal e$ contains only the scalar multiples of $e$, we see that $y=\mu e$ and $z=\nu e$ for some $\mu,\nu\in(0,1]$. Consequently, $y\preceq z$ or $z\preceq y$.

In order to prove that (ii) implies (i), consider the abelian \AW{} \Ncal{} generated by $x$. Clearly, $\Ncal\simeq C(X)$, where $X$ is a Stonean space. Let $f\in C(X)$ correspond to $x$. Suppose that $f$ attains two different nonzero values at points $t_1, t_2\in X$. Take disjoint clopen subsets $O_1$ and $O_2$ of $X$ such that $t_1\in O_1$ and $t_2\in O_2$. Consider $g_1=\chi_{O_1} f$ and $g_2=\chi_{O_2} f$, where $\chi_{O_1}$ and $\chi_{O_2}$ are characteristic functions of $O_1$ and $O_2$, respectively. Then $g_1\not\leq g_2$ and $g_2\not\leq g_1$ which is a contradiction with (ii). Therefore, $f=\la\chi_O$ for some $\la\in(0,1]$, where $\chi_O$ is a characteristic function of a clopen set $O$. Hence there is a nonzero projection $e\in P(\Mcal)$ such that $x=\la e$. Suppose that $e$ is not atomic. Then one can find a nonzero projection $p\in P(\Mcal)$ such that $p\leq e$ and $p\neq e$. This implies that $\la p\preceq \la e$ and $\la(e-p)\preceq \la e$. However, $\la(e-p)\not\preceq \la p$ and $\la p\not\preceq \la (e-p)$ which is a contradiction with (ii). Thus $e$ is atomic projection.
\end{proof}

Recall the notion of a distributive element in a lattice. Let $(P,\leq)$ be a lattice. We say that an element $z\in P$ is {\it distributive} if 
$$
z\vee (x\wedge y)=(z\vee x)\wedge(z\vee y)
$$
for all $x,y\in P$. We denote by $\Dcal(P,\leq)$ the set of all distributive elements in $(P,\leq)$. In the next proposition, we characterize some classes of central elements in terms of distributivity. It turns out that this observation is essential to prove various results about spectral order isomorphisms.

\begin{pro}\label{central operators}
	Let \Mcal{} be an \AW{}. Then
	\begin{enumerate}
		\item $\Zcal(\Mcal)_{sa}=\Dcal(\Mcal_{sa},\preceq)$;
		\item $\Zcal(\Mcal)_+=\Dcal(\Mcal_+,\preceq)$;
		\item $\Ecal(\Zcal(\Mcal))=\Dcal(\Ecal(\Mcal),\preceq)$.
	\end{enumerate}
\end{pro}
\begin{proof}
We shall prove only the statement (i). Proofs of remaining assertions (ii) and (iii) are similar. Let $z\in \Zcal(\Mcal)_{sa}$. Then $E^{z}_\la\in \Zcal(\Mcal)$ for all $\la\in\rr$. Take $x,y\in \Mcal_{sa}$. Applying Proposition~\ref{characterizations of central projections},
	\begin{eqnarray*}
		E^{z\vee(x\wedge y)}_\la &=& E^{z}_\la\wedge E^{x\wedge y}_\la
		= E^{z}_\la\wedge \bigwedge_{\mu>\la} (E^{x}_\mu\vee E^{y}_\mu)
		= \bigwedge_{\mu>\la} E^{z}_\mu\wedge  (E^{x}_\mu\vee E^{y}_\mu)\\
		&=&\bigwedge_{\mu>\la} (E^{z}_\mu\wedge E^{x}_\mu)\vee (E^{z}_\mu\wedge E^{y}_\mu)
		= \bigwedge_{\mu>\la} E^{z\vee x}_\mu\vee E^{z\vee y}_\mu
		=E^{(z\vee x)\wedge(z\vee y)}_\la
	\end{eqnarray*}
for all $\la\in\rr$. Hence 
$$
z\vee (x\wedge y)=(z\vee x)\wedge(z\vee y)
$$
and so $\Zcal(\Mcal)_{sa}\subseteq\Dcal(\Mcal_{sa},\preceq)$.
	
It remains to prove the reverse inclusion. Assume that $z\in\Dcal(\Mcal_{sa},\preceq)$ is nonzero and that $p\in P(\Mcal)$. It is easy to see that
$$
E^{z}_\la=(E^{z}_\la\wedge p)\vee [E^{z}_\la\wedge (\unit-p)]
$$
whenever $\la<-\norm{z}$ and $\la\geq \norm{z}$. As $z\in\Dcal(\Mcal_{sa},\preceq)$ and
$$
\norm{z}(2p-\unit)\wedge \norm{z}(\unit-2p)=-\norm{z}\unit,
$$
we have
$$
z=z\vee \left[\norm{z}(2p-\unit)\wedge \norm{z}(\unit-2p)\right]=\left(z\vee \norm{z}(2p-\unit)\right)\wedge\left(z\vee \norm{z}(\unit-2p)\right).
$$
Hence 
	\begin{eqnarray*}
		E^{z}_\la &=& E^{\left(z\vee \norm{z}(2p-\unit)\right)\wedge\left(z\vee \norm{z}(\unit-2p)\right)}_\la 
		=\bigwedge_{\mu>\la} E^{\left(z\vee \norm{z}(2p-\unit)\right)}_\mu \vee E^{\left(z\vee \norm{z}(\unit-2p)\right)}_\mu\\
		&=& \bigwedge_{\mu>\la} (E^{z}_\mu\wedge E^{\norm{z}(2p-\unit)}_\mu) \vee (E^{z}_\mu\wedge E^{\norm{z}(\unit-2p)}_\mu)
	\end{eqnarray*}
for all $\la\in\rr$. Since $(E^{z}_\mu\wedge E^{\norm{z}(2p-\unit)}_\mu)_{\mu\in\rr}$ is an increasing net, $E^{z}_\mu\wedge E^{\norm{z}(2p-\unit)}_\mu$ commutes with $E^{z}_\nu\wedge E^{\norm{z}(2p-\unit)}_\nu$ for all $\mu,\nu\in\rr$. In the same way, we can show that $E^{z}_\mu\wedge E^{\norm{z}(\unit-2p)}_\mu$ commutes with $E^{z}_\nu\wedge E^{\norm{z}(\unit-2p)}_{\nu}$ for all $\mu,\nu\in\rr$. Moreover, if $\mu,\nu\in[-\norm{z},\norm{z})$, then $E^{z}_\mu\wedge E^{\norm{z}(2p-\unit)}_\mu$ and $E^{z}_\nu\wedge E^{\norm{z}(\unit-2p)}_\nu$ commute because they are mutually orthogonal. Therefore, 
$$
\set{E^{z}_\mu\wedge E^{\norm{z}(2p-\unit)}_\mu}{\mu\in[-\norm{z},\norm{z})}\cup\set{E^{z}_\mu\wedge E^{\norm{z}(\unit -2p)}_\mu}{\mu\in[-\norm{z},\norm{z})}
$$
is contained in an abelian \AWsa{} of \Mcal{}. If $\la\in[-\norm{z},\norm{z})$ and $I_\la=(\la,\norm{z})$, then we see from Lemma~\ref{supremum of infima} that 
	\begin{eqnarray*}
		E^{z}_\la &=& \bigwedge_{\mu>\la} (E^{z}_\mu\wedge E^{\norm{z}(2p-\unit)}_\mu) \vee (E^{z}_\mu\wedge E^{\norm{z}(\unit-2p)}_\mu)\\
		&=& \bigwedge_{\mu\in I_\la} (E^{z}_\mu\wedge E^{\norm{z}(2p-\unit)}_\mu) \vee (E^{z}_\mu\wedge E^{\norm{z}(\unit-2p)}_\mu)\\
		&=& \left(\bigwedge_{\mu\in I_\la} E^{z}_\mu\wedge E^{\norm{z}(2p-\unit)}_\mu\right) \vee \left(\bigwedge_{\mu\in I_\la} E^{z}_\mu\wedge E^{\norm{z}(\unit-2p)}_\mu\right)\\
		&=& \left[\left(\bigwedge_{\mu\in I_\la} E^{z}_\mu\right)\wedge \left(\bigwedge_{\mu\in I_\la} E^{\norm{z}(2p-\unit)}_\mu\right)\right] \vee \left[\left(\bigwedge_{\mu\in I_\la} E^{z}_\mu\right)\wedge \left(\bigwedge_{\mu\in I_\la} E^{\norm{z}(\unit-2p)}_\mu\right)\right]\\
		&=&(E^{z}_\la\wedge E^{\norm{z}(2p-\unit)}_\la) \vee (E^{z}_\la\wedge E^{\norm{z}(\unit-2p)}_\la)
		=(E^{z}_\la\wedge E^{p}_\frac{\la+\norm{z}}{2\norm{z}}) \vee (E^{z}_\la\wedge E^{\unit-p}_\frac{\la+\norm{z}}{2\norm{z}})\\
		&=&[E^{z}_\la\wedge (\unit-p)] \vee (E^{z}_\la\wedge p).
	\end{eqnarray*}
According to Proposition~\ref{characterizations of central projections}, $E^{z}_\la$ is a central projection for each $\la\in\rr$. As each projection in the spectral family of $z$ is central, $z\in\Zcal(\Mcal)$.
\end{proof}
\section{Spectral order isomorphisms}

It was proved in \cite[Theorem~4.3]{HT17} that canonical spectral order automorphisms are characterized by preserving scalar multiples of projections. We generalize this result to spectral order isomorphisms.

\begin{theo}\label{canonical spectral order isomorphisms generalization}
	Let \Mcal{} and \Ncal{} be \AW{}s. Suppose that $\f:\Ecal(\Mcal)\to\Ecal(\Ncal)$ is a spectral order isomorphism. Then the following statements are equivalent:
	\begin{enumerate}
		\item For each $p\in P(\Mcal)$, $\f(p)\in P(\Ncal)$ and there is a strictly increasing bijection $f_p:[0,1]\to[0,1]$ such that $\f(\la p)=f_p(\la)\f(p)$ for all $\la\in[0,1]$.
		\item \f{} is canonical.
	\end{enumerate}
\end{theo}
\begin{proof}
	$(i)\Rightarrow (ii)$. Define the map $\tau: P(\Mcal)\to P(\Ncal)$ by $\tau(p)=\f(p)$. The map $\tau$ is a projection isomorphism. Thus $\Theta_{\tau}^{-1}\circ\f:\Ecal(\Mcal)\to\Ecal(\Mcal)$ is a spectral order isomorphism. By \cite[Theorem~4.3]{HT17}, $\Theta_{\tau}^{-1}\circ\f$ is canonical. Therefore, $\f$ is canonical.
	
	$(ii)\Rightarrow (i)$. By the assumption, there are a strictly increasing bijection $f:[0,1]\to[0,1]$ and a projection isomorphism $\tau: P(\Mcal)\to P(\Ncal)$ such that $\f(x)=\Theta_\tau(f(x))$ for all $x\in\Ecal(\Mcal)$. A direct computation shows that 
	$$
	\f(\la p)=\Theta_{\tau}(f(\la)p)=f(\la)\Theta_{\tau}(p)=f(\la)\f(p)
	$$
	for all $\la\in[0,1]$.
\end{proof}
 
Using the previous theorem, we obtain in Corollary~\ref{spectral order isomorphisms are canonical} that every spectral order isomorphism between sets of all effects in \AWf{}s of Type I is canonical. This statement is a slight generalization of known results concerning spectral order automorphisms \cite{MN16,MS07}. In contrast with the form of automorphisms in \cite{MN16,MS07}, it shows that the main structure of a spectral order isomorphism between sets of all effects does not depend on a particular choice of \AWf{}s of Type~I. We give the proof of Corollary~\ref{spectral order isomorphisms are canonical} for completeness of discussion. Note that an advantage of our approach is that we need not discuss the case of \AWf{} of Type I$_2$ separately.

\begin{cor}\label{spectral order isomorphisms are canonical}
	If \Mcal{} and \Ncal{} are \AWf{}s of Type $I$, then every spectral order isomorphism $\f:\Ecal(\Mcal)\to\Ecal(\Ncal)$ is canonical.
\end{cor}
\begin{proof}
It follows from Proposition~\ref{characterization of atomic projections on effects} that if $p\in P_{at}(\Mcal)$ and $\la\in(0,1]$, then there are $q\in P_{at}(\Ncal)$ and $\mu\in(0,1]$ such that $\f(\la p)=\mu q$. Take $\la_1,\la_2\in (0,1]$ and $p\in P_{at}(\Mcal)$. Then $\f(\la_1 p)=\mu_1 q_1$ and $\f(\la_2 p)=\mu_2 q_2$ for some $\mu_1,\mu_2\in(0,1]$ and $q_1,q_2\in P_{at}(\Ncal)$. Let us show that $q_1=q_2$. It is not difficult to verify that $\la_1 p\vee \la_2 p=\max\{\la_1,\la_2\}p$. Thus 
$$
\f(\max\{\la_1,\la_2\}p)=\f(\la_1 p\vee \la_2 p)=\mu_1 q_1\vee \mu_2 q_2. 
$$
Suppose that $q_1\neq q_2$. If $\mu_1q_1\preceq \mu_2q_2$, then $E^{\mu_2q_2}_{\nu}\leq E^{\mu_1q_1}_{\nu}$ for every $\nu\in\rr$. In particular, $E^{\mu_2q_2}_{0}\leq E^{\mu_1q_1}_{0}$ and so $q_1\leq q_2$. This is a contradiction because $q_1$ and $q_2$ are two distinct atoms. Hence $\mu_1q_1\not\preceq \mu_2q_2$. Interchanging the roles of $\mu_1q_1$ and $\mu_2q_2$, we see that $\mu_2q_2\not\preceq \mu_1q_1$. According to Proposition~\ref{characterization of atomic projections on effects}, there are no $\mu\in(0,1]$ and $e\in P_{at}(\Ncal)$ such that $\f(\max\{\la_1,\la_2\}p)=\mu e$ which is a contradiction. Consequently, $q_1=q_2$. As $\f^{-1}$ is also a spectral order isomorphism, there are a strictly increasing bijection $f_p:[0,1]\to[0,1]$ and $q\in P_{at}(\Ncal)$ such that $\f(\la p)=f_p(\la)q$ for all $\la\in[0,1]$. 
	
If $p,q\in P_{at}(\Mcal)$ are different and $\la\in(0,1]$, then there is $e\in P_{at}(\Mcal)$ different from $p$ and $q$ such that $e\leq p\vee q$. Multiplication by $\la$ is a spectral order automorphism of $\Mcal_+$ which maps $\Ecal(\Mcal)$ into itself. This implies that
$$
\f(\la p\vee\la q)=\f(\la p\vee\la e)=\f(\la e\vee\la q).
$$
Hence
$$
f_p(\la)\f(p)\vee f_q(\la)\f(q)=f_p(\la)\f(p)\vee f_e(\la)\f(e)=f_e(\la)\f(e)\vee f_q(\la)\f(q).
$$
Since the spectra of these three operators have to coincide, we see from Lemma~\ref{supremum of multiples of projections} that $f_p(\la)=f_e(\la)=f_q(\la)$. Therefore, there is a strictly increasing function $f:[0,1]\to[0,1]$ such that $\f(\la p)=f(\la)\f(p)$ for all $p\in P_{at}(\Mcal)$ and $\la\in[0,1]$. As every projection in \Mcal{} is the supremum of a set of atomic projections in \Mcal{}, $\f(\la p)=f(\la)\f(p)$ for all $p\in P(\Mcal)$ and $\la\in[0,1]$. According to Theorem~\ref{canonical spectral order isomorphisms generalization}, \f{} is canonical.
\end{proof}

Now we would like to extend the preceding corollary to positive operators and also to self-adjoint operators. We start with the case of positive operators. For this, we need the following lemmas.

\begin{lem}\label{Image of the set of all effects}
	Let \Mcal{} and \Ncal{} be \AW{}s and let $\al\in[1,\infty)$. If $f:[0,\infty)\to[0,\infty)$ is a strictly increasing bijection and $\f:\Mcal_+\to\Ncal_+$ is a spectral order isomorphism such that $\f(\la\unit_\Mcal)=f(\la) \unit_\Ncal$ for all $\la\in[0,\infty)$, then 
	$$
	\frac{1}{\al}\f(f^{-1}(\al \Ecal(\Mcal)))=\Ecal(\Ncal).
	$$
\end{lem}
\begin{proof}
	Let $x\in\Ecal(\Mcal)$. Lemma~\ref{basic properties of specral order} establishes that $0\preceq x\preceq \unit_\Mcal$. Since multiplication by a positive number is a spectral order isomorphism, we have 
	$$
	0\preceq \frac{1}{\al}\f(f^{-1}(\al x))\preceq \frac{1}{\al}\f(f^{-1}(\al \unit_\Mcal{})) 
				= \unit_\Ncal.
	$$
	By Lemma~\ref{basic properties of specral order}, $\frac{1}{\al}\f(f^{-1}(\al x))\in\Ecal(\Ncal)$.
	
	Now take $y\in\Ecal(\Ncal)$. Consider $x=\frac{1}{\al}f(\f^{-1}(\al y))$. Similar arguments as above shows that $x\in\Ecal(\Mcal)$. Moreover, $\frac{1}{\al}\f(f^{-1}(\al x))=y$.
\end{proof}

\begin{lem}\label{Extension from effects to positive operators}
	Let $\f:\Mcal_{+}\to\Ncal_+$ be a spectral order isomorphism between positive parts of \AW{}s \Mcal{} and \Ncal{} such that there exists a strictly increasing bijection $f:[0,\infty)\to[0,\infty)$ satisfying $\f(\la\unit_\Mcal)=f(\la)\unit_\Ncal$ for all $\la\in[0,\infty)$.
	Suppose that, for each $\al\in[1,\infty)$, the equation 
	$$
	\f_\al(x)=\frac{1}{\al}\f(f^{-1}(\al x))
	$$
	defines a canonical spectral order isomorphism $\f_\al:\Ecal(\Mcal)\to\Ecal(\Ncal)$. 
	\begin{enumerate}
		\item There exists a projection isomorphism $\tau:P(\Mcal)\to P(\Ncal)$ such that $\f(x)=\Theta_\tau(f(x))$ for all $x\in\Mcal_+$. 
		\item If, in addition, there is a \Jsi{} $\psi:\Mcal \to\Ncal$ such that $\psi(x)=\f(f^{-1}(x))$ for all $x\in\Ecal(\Mcal)$, then $\f(x)=\psi(f(x))$ for all $x\in\Mcal_+$.
	\end{enumerate}
\end{lem}
\begin{proof}
	\begin{enumerate}
		\item Maps $\f_\al$ are canonical spectral order isomorphisms and so, for each $\al\in[1,\infty)$, there are a strictly increasing bijection $g_\al:[0,1]\to[0,1]$ and a projection isomorphism $\tau_\al: P(\Mcal)\to P(\Ncal)$ such that $\f_\al(x)=\Theta_{\tau_\al}(g_\al(x))$ whenever $x\in \Ecal(\Mcal)$.
		
		Let $\al\in[1,\infty)$ and $\la\in[0,1]$ be arbitrary. Then 
		$$
		\al g_\al(\la)\unit_\Ncal=\al\f_\al(\la\unit_\Mcal)=\f(f^{-1}(\al\la\unit_\Mcal))=\al\la\unit_\Ncal.
		$$
		Thus $g_\al(\la)=\la$ and so $\f_\al(x)=\Theta_{\tau_\al}(x)$ for all $x\in\Ecal(\Mcal)$. Set $\tau=\tau_1$. If $0\leq x\leq \frac{1}{\al}\unit_\Mcal$, then
		$$
		\Theta_{\tau_\al}(\al x)=\al\Theta_{\tau_\al}(x)=\al\f_\al(x)=\f(f^{-1}(\al x))=\Theta_\tau(\al x).
		$$
		Consequently, $\Theta_{\tau_\al}=\Theta_\tau$ on $\Mcal_+$.
		
		If $y\in\Mcal_+$, then there are $x\in\Ecal(\Mcal)$ and $\beta\in[1,\infty)$ such that $y=\beta x$. Hence
		$$
		\f(f^{-1}(y))=\f(f^{-1}(\beta x))=\beta\Theta_{\tau_\beta}(x)=\beta\Theta_{\tau}(x)=\Theta_\tau(y).
		$$
		
		\item According to the proof of (i), $\Theta_\tau(x)=\psi(x)$ for all $x\in\Ecal(\Mcal)$. Hence $\Theta_\tau$ and $\psi$ coincide on $\Mcal_+$.
	\end{enumerate}
\end{proof}

\begin{theo}\label{spectral order isomorphisms are canonical II}
	If \Mcal{} and \Ncal{} are \AWf{}s of Type $I$, then every spectral order isomorphism $\f:\Mcal_+\to\Ncal_+$ is canonical.
\end{theo}
\begin{proof}
	By Proposition~\ref{central operators}, \f{} preserves central positive operators. Thus there is a strictly increasing bijection $f:[0,\infty)\to[0,\infty)$ such that $\f(\la\unit)=f(\la)\unit$ for every $\la\in[0,\infty)$. Let $\al\in[1,\infty)$. Set $\f_\al(x)=\frac{1}{\al}\f(f^{-1}(\al x))$ for every $x\in\Ecal(\Mcal)$. It follows from Lemma~\ref{Image of the set of all effects} and Corollary~\ref{spectral order isomorphisms are canonical} that $\f_\al$ is a canonical spectral order isomorphism from $\Ecal(\Mcal)$ onto $\Ecal(\Ncal)$. Therefore, we obtain from Lemma~\ref{Extension from effects to positive operators} that \f{} is canonical.
\end{proof}

\begin{lem}\label{Extension from positive to self-adjoint operators}
	Let $\f:\Mcal_{sa}\to\Mcal_{sa}$ be a spectral order automorphism of the self-adjoint part of an \AW{} \Mcal{} such that 
	\begin{enumerate}
		\item $\f(\la\unit)=\la\unit$ for all $\la\in\rr$;
		\item $\f(x)=x$ for all $x\in\Mcal_+$.
	\end{enumerate}
	If, for each $\al\in(0,\infty)$, the equation 
	$$
	\f_\al(x)=\f(x-\al\unit)+\al\unit
	$$
	defines a canonical spectral order automorphism $\f_\al:\Mcal_+\to\Mcal_+$, then $\f(x)=x$ for every $x\in\Mcal_{sa}$.
\end{lem}
\begin{proof}
	Let $x\in\Mcal_{sa}$. Then there exists $\al\in(0,\infty)$ such that $x+\al\unit$ is positive invertible element. Hence $\beta\unit\leq x+\al\unit$ for some $\beta\in(0,\infty)$ and so $\al\unit\leq \frac{\al}{\beta}(x+\al\unit)$. As $\f$ is the identity map on positive operators, we have 
	$$
	\frac{\al}{\beta}(x+\al\unit)=\f\left(\frac{\al}{\beta}(x+\al\unit)-\al\unit\right)+\al\unit=\f_\al\left(\frac{\al}{\beta}(x+\al\unit)\right).
	$$	
	Moreover, $\f_\al(x)=\Theta_{\tau_\al}(x)$ for some projection isomorphism $\tau_\al:P(\Mcal)\to P(\Mcal)$ because $\f_\al$ is canonical and $\f_\al(\la \unit)=\la \unit$ for all $\la\in[0,\infty)$. Hence 
	$$
	\frac{\al}{\beta}(x+\al\unit)=\Theta_{\tau_\al}\left(\frac{\al}{\beta}(x+\al\unit)\right)=\frac{\al}{\beta}\Theta_{\tau_\al}\left(x+\al\unit\right)
	=\frac{\al}{\beta}\f_{\al}\left(x+\al\unit\right).
	$$
	Consequently, 
	$$
	x=\f_\al(x+\al\unit)-\al\unit=\f(x).
	$$
\end{proof}

\begin{theo}\label{spectral order isomorphisms are canonical III}
	If \Mcal{} and \Ncal{} are \AWf{}s of Type $I$, then every spectral order isomorphism $\f:\Mcal_{sa}\to\Ncal_{sa}$ is canonical.
\end{theo}
\begin{proof}
	According to Proposition~\ref{central operators}, $\f(\la\unit_\Mcal)=f(\la)\unit_\Ncal$ for all $\la\in\rr$, where $f:\rr\to\rr$ is a strictly increasing bijection $f:\rr\to\rr$. We may assume without loss of generality that $f$ is identity on $\rr$. Then $f(0)=0$ and so $\f(0)=0$. Hence $\f(\Mcal_+)=\Ncal_+$. Theorem~\ref{spectral order isomorphisms are canonical II} ensures the existence of a projection isomorphism $\tau: P(\Mcal)\to P(\Ncal)$ such that $\f(x)=\Theta_\tau(x)$ for all $x\in\Mcal_+$. 
	
	Set $\psi(x)=\f(\Theta^{-1}_\tau(x))$, $x\in\Ncal_{sa}$. Then $\psi:\Ncal_{sa}\to\Ncal_{sa}$ is a spectral order automorphism satisfying $\psi(x)=x$ for all $x\in\Ncal_{+}$ and $\psi(\la\unit)=\la\unit$ for all $\la\in\rr$. Take $\al\in(0,\infty)$ and put
	$$
	\psi_\al(x)=\psi(x-\al\unit)+\al\unit
	$$
	for each $x\in\Ncal_+$. It is easy to see that $\psi_\al:\Ncal_+\to\Ncal_+$ is a spectral order automorphism. By Theorem~\ref{spectral order isomorphisms are canonical II}, $\psi_\al$ is canonical. From Lemma~\ref{Extension from positive to self-adjoint operators}, $\psi(x)=x$ for all $x\in \Ncal_{sa}$. This completes the proof.
\end{proof}
\section{Spectral order orthoisomorphisms}

Let $\Mcal{}$ be an \AW{} having no Type I$_2$ direct summand. In \cite[Theorem~5]{HT16}, Hamhalter and Turilova proved that every spectral order orthoautomorphism $\f$ on $\Ecal(\Mcal)$ is given by the continuous function calculus and a \Jsi{} provided that $\f$ preserves scalar multiples of the unit. In the following theorem, we reformulate this result for spectral order orthoisomorphisms between sets of all effects. We present a proof here for the sake of completeness.

\begin{theo}\label{spectral order orthoisomorphisms preserving multiples of unit generalization}
	Let \Mcal{} be an \AW{} having no Type I$_2$ direct summand and let \Ncal{} be an \AW{}. If a spectral order orthoisomorphism $\f:\Ecal(\Mcal)\to\Ecal(\Ncal)$ and a strictly increasing bijection $f:[0,1]\to[0,1]$ satisfy $\f(\la\unit_\Mcal)=f(\la)\unit_\Ncal$ for all $\la\in[0,1]$, then there is a unique \Jsi{} $\psi:\Mcal\to\Ncal$ such that 
	$$
	\f(x)=\psi(f(x))
	$$
	for all $x\in\Ecal(\Mcal)$.
\end{theo}
\begin{proof}
	Let $p$ be a projection in \Mcal{}. Since \f{} preserves orthogonality, $\f(p)$ is orthogonal to $\f(\unit-p)$. Hence 
	$$
	\unit=\f(\unit)=\f(p\vee(\unit-p))=\f(p)\vee\f(\unit-p)=\f(p)+\f(\unit-p),
	$$
	where we have used Lemma~\ref{supremum of orthogonal elements}. Multiplying both sides by $\f(p)$, we see that $\f(p)$ is a projection. Consequently, $\f(P(\Mcal))=P(\Ncal)$ and so \f{} restricts to an orthoisomorphism from $P(\Mcal)$ onto $P(\Ncal)$. By Dye's theorem (see Theorem~\ref{Dye theorem}), there is a \Jsi{} $\varrho:\Mcal\to\Ncal$ such that \f{} and $\varrho$ coincide on $P(\Mcal)$. Consider $\theta=\varrho^{-1}\circ\f$. The map $\theta$ is a spectral order orthoautomorphism from $\Ecal(\Mcal)$ onto itself such that $\theta(\la\unit_\Mcal)=f(\la)\unit_\Ncal$ for all $\la\in[0,1]$. According to \cite[Theorem~5]{HT16}, we obtain the required conclusion.
\end{proof}

In the case of an \AWf{} \Mcal{} that is not of Type I$_2$ and Type III, it was proved by Hamhalter and Turilova \cite{HT17} that every spectral order orthoautomorphism $\f$ on $\Ecal(\Mcal)$ automatically preserves scalar multiples of the unit and so we have a complete description of such automorphisms. However, the question of Type III factors was left open in \cite{HT17}. It turns out that the case of Type III factors does not lead to any different behavior.

\begin{cor}\label{spectral order orthoisomorphisms between sets of effects in factors}
	Let \Mcal{} and \Ncal{} be \AWf{}s that are not of Type I$_2$. Let $\f:\Ecal(\Mcal)\to\Ecal(\Ncal)$ be a spectral order orthoisomorphism. Then there are a unique strictly increasing bijection $f:[0,1]\to [0,1]$ and a unique \Jsi{} $\psi:\Mcal\to\Ncal$ such that 
	$$
	\f(x)=\psi(f(x))
	$$
	for all $x\in\Ecal(\Mcal)$.
\end{cor}
\begin{proof}
	The statement is a direct consequence of Proposition~\ref{central operators} and Theorem~\ref{spectral order orthoisomorphisms preserving multiples of unit generalization}.
\end{proof}

The previous results can be generalized to spectral order isomorphisms between positive parts (or between self-adjoint parts) of \AW{}s as follows.

\begin{theo}\label{spectral order orthoisomorphisms on positive operators}
	Let \Mcal{} be an \AW{} having no Type I$_2$ direct summand and let \Ncal{} be an \AW{}. If a spectral order orthoisomorphism $\f:\Mcal_+\to\Ncal_+$ and a strictly increasing bijection $f:[0,\infty)\to [0,\infty)$ satisfy $\f(\la\unit_\Mcal)=f(\la)\unit_\Ncal$ for all $\la\in[0,\infty)$, then there is a unique \Jsi{} $\psi:\Mcal\to\Ncal$ such that 
	$$
	\f(x)=\psi(f(x))
	$$
	for all $x\in\Mcal_+$.
\end{theo}
\begin{proof}
	Let $\al\in[1,\infty)$. We see from Lemma~\ref{Image of the set of all effects} that 
	$$
	\f_\al(x)=\frac{1}{\al}\f(f^{-1}(\al x))
	$$ 
	defines a spectral order orthoisomorphism from $\Ecal(\Mcal)$ onto $\Ecal(\Ncal)$. It follows from Theorem~\ref{spectral order orthoisomorphisms preserving multiples of unit generalization} that $\f_\al$ coincides with a \Jsi{} on $\Ecal(\Mcal)$. Thus Lemma~\ref{Extension from effects to positive operators} ensures the existence of a \Jsi{} $\psi:\Mcal\to\Ncal$ such that $\f(x)=\psi(f(x))$ for all $x\in\Mcal_+$.
\end{proof}

\begin{cor}
	Let \Mcal{} and \Ncal{} be \AWf{}s that are not of Type I$_2$. Let $\f:\Mcal_+\to\Ncal_+$ be a spectral order orthoisomorphism. Then there are a unique strictly increasing bijection $f:[0,\infty)\to [0,\infty)$ and a unique \Jsi{} $\psi:\Mcal\to\Ncal$ such that 
	$$
	\f(x)=\psi(f(x))
	$$
	for all $x\in\Mcal_+$.
\end{cor}
\begin{proof}
	It follows from Proposition~\ref{central operators} and Theorem~\ref{spectral order orthoisomorphisms on positive operators}.
\end{proof}

Let us fix a notation. We denote by $x^+$ and $x^-$, respectively, the positive part and the negative part of a self-adjoint operator $x$ in an \AW{}.

\begin{lem}\label{positive and negative parts}
	Let $\f:\Mcal_{sa}\to\Ncal_{sa}$ be a spectral order isomorphism between self-adjoint parts of \AW{}s \Mcal{} and \Ncal{} with $\f(0)=0$. If $x\in\Mcal_{sa}$, then $\f(x)^+=\f(x^+)$ and $\f(x)^-=-\f(-(x^-))$.
\end{lem}
\begin{proof}
	Using Lemma~\ref{supremum of orthogonal elements} and the fact that multiplication by $-1$ is order-reversing map, $y\vee 0=y^+$ and $y\wedge 0=-y^-$ for each $y\in\Mcal_{sa}$. Hence $\f(x)^+=\f(x^+)$ and $\f(x)^-=-\f(-(x^-))$.
\end{proof}

\begin{theo}\label{spectral order orthoisomorphisms on self-adjoint operators}
	Let \Mcal{} be an \AW{} having no Type I$_2$ direct summand and let \Ncal{} be an \AW{}. If a spectral order orthoisomorphism $\f:\Mcal_{sa}\to\Ncal_{sa}$ and a strictly increasing bijection $f:\rr\to\rr$ satisfy $\f(\la\unit_\Mcal)=f(\la)\unit_\Ncal$ for all $\la\in\rr$, then there is a unique \Jsi{} $\psi:\Mcal\to\Ncal$ such that 
	$$
	\f(x)=\psi(f(x))
	$$
	for all $x\in\Mcal_{sa}$.
\end{theo}
\begin{proof}
	As $\f$ preserves orthogonality and $0$ is orthogonal to $\la\unit_\Mcal$ for all $\la\in\rr$, we have $\f(0)=f(0)\unit_\Ncal=0$.	We infer from this that $\f(\Mcal_+)=\Ncal_+$. By Theorem~\ref{spectral order orthoisomorphisms on positive operators}, there is a unique \Jsi{} $\psi:\Mcal\to\Ncal$ such that $\f(x)=\psi(f(x))$ for all $x\in\Mcal_+$. 
	
	Consider a spectral order orthoautomorphism $\varrho:\Ncal_{sa}\to\Ncal_{sa}$ defined by 
	$$
	\varrho(x)=-\f(f^{-1}(\psi^{-1}(-x))).
	$$
	In order to complete the proof, we need to show that $\varrho(x)=x$ for all $x\in\Ncal_{sa}$. It follows from the above discussion that $\varrho$ is the identity map on $\Ncal_-=-\Ncal_+$, $\varrho(\Ncal_+)=\Ncal_+$, and $\varrho(\la\unit)=\la\unit$ for all $\la\in\rr$. By virtue of Theorem~\ref{spectral order orthoisomorphisms on positive operators}, there exists a unique \Jsi{} $\widetilde{\psi}:\Ncal\to\Ncal$ such that $\varrho(x)=\widetilde{\psi}(x)$ for all $x\in\Ncal_+$. Let $p\in P(\Ncal)$. Entirely similar discussion to that of the proof of Theorem~\ref{spectral order orthoisomorphisms preserving multiples of unit generalization} leads to the equation $\unit=\rho(p)+\rho(\unit-p)$. Upon multiplying both sides of the equation by $\varrho(-p)$, we get
	$$
	\varrho(-p)=\varrho(p)\varrho(-p).
	$$
	Moreover, $\varrho$ is the identity map on $\Ncal_-$ and so 
	$$
	-\unit=\varrho(-p)+\varrho(-(\unit-p)).
	$$
	We multiply both sides of the last equation by $\varrho(p)$ to obtain
	$$
	-\varrho(p)=\varrho(p)\varrho(-p)=\varrho(-p).
	$$
	Hence $p=\widetilde{\psi}(p)$. As $p\in P(\Ncal)$ is arbitrary, $\widetilde{\psi}(x)=x$ for all $x\in\Ncal$. This means that $\varrho(x)=x$ for all $x\in\Ncal_+\cup\Ncal_-$. 
	
	Now we assume that $x\in\Ncal_{sa}$. We see from Lemma~\ref{positive and negative parts} that
	$$
	\varrho(x)=\varrho(x)^+-\varrho(x)^-=\varrho(x^+)+\varrho(-(x^-))=x^+-x^-=x.
	$$
\end{proof}

\begin{cor}
	Let \Mcal{} and \Ncal{} be \AWf{}s that are not of Type I$_2$. Let $\f:\Mcal_{sa}\to\Ncal_{sa}$ be a spectral order orthoisomorphism. Then there are a unique strictly increasing bijection $f:\rr\to \rr$ and a unique \Jsi{} $\psi:\Mcal\to\Ncal$ such that 
	$$
	\f(x)=\psi(f(x))
	$$
	for all $x\in\Mcal_{sa}$.
\end{cor}
\begin{proof}
	The statement is a direct consequence of Proposition~\ref{central operators} and Theorem~\ref{spectral order orthoisomorphisms on self-adjoint operators}.
\end{proof}

\section*{Acknowledgement}
This work was supported by the project OPVVV Center for Advanced Applied Science CZ.02.1.01/0.0/0.0/16\_019/0000778 and  the ``Grant Agency of the Czech Republic" grant number 17-00941S, ``Topological and geometrical properties of Banach spaces and operator algebras II". 


\end{document}